\documentclass[dvips,preprint]{imsart}
\usepackage[latin1]{inputenc}
\usepackage{amsmath, amssymb, amsthm}
\usepackage[dvips]{graphicx}

\startlocaldefs
\newcommand{\cov}{\operatorname{Cov}}

\newcommand{\Id}{\operatorname{Id}}

\newtheorem{theo}{Theorem}[section]

\newtheorem{lem}{Lemma}[section]

\endlocaldefs

\begin{document}

\begin{frontmatter}

\author{Julien Poisat}
\runauthor{Poisat}
\runtitle{Pinning with Markov disorder}

\title{Two Pinning Models with Markov disorder}

\maketitle

\begin{abstract}
Disordered pinning models deal with the (de)localization transition of a polymer in interaction with a heterogeneous interface. In this paper, we focus on two models where the inhomogeneities at the interface are not independent but given by an irreducible Markov chain on a finite state space. In the first model, using Markov renewal tools, we give an expression for the annealed critical curve in terms of a Perron-Frobenius eigenvalue, and provide examples where exact computations are possible. In the second model, the transition matrix vary with the size of the system so that, roughly speaking, disorder is more and more correlated. In this case we are able to give the limit of the averaged quenched free energy, therefore providing the full phase diagram picture, and the number of critical points is related to the number of states of the Markov chain. We also mention that the question of pinning in correlated disorder appears in the context of DNA denaturation.
\end{abstract}

\begin{keyword}
\kwd pinning models
\kwd statistical mechanics
\kwd polymers
\kwd disordered systems
\kwd annealed system
\kwd  Markov chains
\kwd Perron-Frobenius
\kwd renewal process
\kwd DNA denaturation
\end{keyword}

\end{frontmatter}

\section{Introduction}

Among statistical mechanics systems, disordered pinning models have received much attention in the last past years. This class of models deals with the localization/delocalization transition of a polymer interacting with an inhomogeneous interface, and whereas the homogeneous version is fully solvable, some questions on the disordered case were answered only recently. This paper deals with two pinning models with Markov disorder, and is organized as follows:  in a first part we will recall notations and general facts on the model, and give an overview of results dealing with the critical curve of the phase diagram (for a more detailed account on pinning models, the reader can refer to the monographs \cite{GG_Book}, \cite[Ch. 7,11]{DenH_Book} and the survey papers \cite{GG_Overview}, \cite{Toninelli_Survey}). We then introduce the two models, which will be developed in one section each.

\subsection{General model and general facts}

The model we use is the renewal pinning model, that we recall here. Suppose that the interface is modelled by the half line $[0,+\infty)$ and the points of contact between the polymer and the interface by a discrete renewal process $\tau = (\tau_n)_{n\geq0}$ where $\tau_0:=0$ and $K(n):= P(\tau_1 = n)=L(n) n^{-(1+\alpha)}$ ($L:\mathbb{N}\mapsto(0,+\infty)$ being a slowly varying function and $\alpha\geq0$) is the law of interarrival times. Without loss of generality, we assume $\tau$ is recurrent ($\sum_{n\geq1}K(n) =1$). We will denote by $\delta_n$ the indicator function of the event $\{n\in\tau\}:=\cup_{k\geq0} \{\tau_k = n\}$.

\par Let now $\omega = (\omega_n)_{n\geq0}$ be a sequence of centered random variables, independent of $\tau$, standing for the inhomogeneities on the interface. We will use the notation $\mathbb{E}$ (resp. $E$) for integration with respect to $\mathbb{P}$, the law of $\omega$ (resp. $P$, the law of $\tau$).

\par
For a typical realization of $\omega$, and parameters $h\in\mathbb{R}$, $\beta\geq0$, we define the hamiltonian at size $N$ as $$H_{N,\beta,h,\omega}(\tau) = \sum_{n=1}^N (\beta\omega_n + h)\delta_n$$ and the corresponding polymer law by $$ \frac{dP_{N,\beta,h,\omega}}{dP}(\tau) = \frac{1}{Z_{N,\beta,h,\omega}} \exp(H_{N,\beta,h,\omega}(\tau)) \delta_N$$ where $$Z_{N,\beta,h,\omega} = E(\exp(H_{N,\beta,h,\omega}(\tau))\delta_N) $$ is called the partition function. We call free energy at size $N$, the (random) function $$F_N(\beta,h) = \frac{1}{N}\log Z_{N,\beta,h,\omega}.$$

 If the sequence $\omega$ is stationary, ergodic and integrable, then one can prove by subadditive arguments that there exists a deterministic quantity $F(\beta,h)\geq0$ (the infinite size quenched free energy) such that for all parameters $(\beta,h)$, $\lim_{N\rightarrow+\infty}F_N(\beta,h) = F(\beta,h)$ in the $L^1(\omega)$ and $\omega$-almost surely (see \cite[Theorem 4.6]{GG_Book}). Then let us define the localized phase $\mathcal{L} = \{(\beta,h), F(\beta,h)>0\}$ and the delocalized phase $\mathcal{D} = \{(\beta,h), F(\beta,h)=0\}$. One can see why by looking at the partial derivative of $F$ with respect to $h$: using convexity arguments, we can write $\partial_hF(\beta,h) = \lim_{N\rightarrow +\infty} \partial_h F_N(\beta,h) = \lim_{N\rightarrow +\infty} E_{N,\beta,h}((1/N)\sum_{n=1}^N \delta_n)$ which is the limiting contact density under the polymer law. Still using the convexity of $F(\beta,h)$, one can prove the existence of a concave curve (called critical curve) $\beta \mapsto h_c(\beta)\in [-\infty,0]$ such that for all $\beta\geq0$, $$(\beta,h)\in\mathcal{L} \iff h>h_c(\beta).$$

\subsection{The critical curve: State of the art}

Let us look at the results available on the critical curve. We already know from the homogeneous pinning model that $h_c(0)=0$ (see \cite[Ch.2]{GG_Book}). Then one can prove that $h_c(\beta)\leq0$, and with a little more work, provided disorder is nondegenerate, that $h_c(\beta)<0$ if $\beta>0$: this means that disorder has a localizing effect (see \cite[Proposition 5.1 and Theorem 5.2]{GG_Book}).

Before giving a lower bound, we assume that $\omega = (\omega_n)_{n\geq0}$ is a sequence of \emph{independent and identically distributed} random variables, such that the moment generating function $\Lambda(\beta) := \mathbb{E}(e^{\beta \omega_0})$ is finite at least on an open interval $[0,c)$. This allows us to give the following lower bound on the critical curve: for all $\beta\in[0,c)$, $$h_c(\beta) \geq h_c^a(\beta) := -\log\Lambda(\beta).$$ This a direct consequence of the following observation: by Jensen's inequality,
\begin{equation}\label{annealed_bound}\mathbb{E}F_N(\beta,h)\leq F_N^a(\beta,h) := \frac{1}{N}\log \mathbb{E}Z_{N,\beta,h,\omega}\end{equation} (note that this inequality is general: we do not require any independence assumption at this point). The quantity $Z^a_{N,\beta,h}:=\mathbb{E}Z_{N,\beta,h,\omega}$ is called annealed partition function. If we now use the fact that the $\omega_n$'s are i.i.d., one gets $$Z^a_{N,\beta,h} = E\left(\exp\left((h + \log\Lambda(\beta)) \sum_{n=1}^N \delta_n\right)\delta_N\right)$$ so that in this case the annealed free energy is nothing but the homogeneous free energy ($\beta=0$) with parameter $h$ shifted by $\log\Lambda(\beta)$. Since the critical point of the homogeneous model is $0$, the critical point of the annealed model is $h_c^a(\beta)=-\log\Lambda(\beta)$.

The question of whether the annealed critical curve $h_c^a$ and the quenched critical curve $h_c$ coincide (as well as the related question of equality of critical exponents) has been the topic of many papers. Several authors (\cite{Alexander_Sido,Fabio_replica,Lacoin_martingale,Cheliotis_DenH}) proved consecutively, with different methods, that if $\alpha$ is in $(0,\frac{1}{2})$, then the equality $h_c(\beta) = h_c^a(\beta)$ holds for small enough $\beta$. The same is true for $\alpha=\frac{1}{2}$ provided $\sum_{n\geq1} \frac{1}{nL(n)^2} < +\infty$, and more recently the result has been proved for $\alpha=0$, the equality being true for all values of $\beta$ in this case (see \cite{Alexander_loop_exp_one}). If $\alpha > 1/2$ or if $\alpha = 1/2$ but under some conditions on $L$ ( for instance, if $L(\infty)=0$, or $L$ is a constant) ,  then disorder is relevant for all $\beta$, meaning that $h_c(\beta)>h_c^a(\beta)$ (see \cite{Alexander_quenched},\cite{Derrida_al_relevance},\cite{GG_T_L_Marginal}). The fact that the critical value of $\alpha$ is $1/2$ is in accordance with a heuristic called Harris criterion (see \cite[p.145-146]{Toninelli_Survey} or \cite[p.116-118]{GG_Book} for example).

As pointed out earlier, these results hold in the case of i.i.d. disorder. In \cite{Poisat_frc}, the author considered a model with \emph{locally} correlated disorder. In this model the $\omega_n$'s constitute a gaussian moving average of finite order (meaning that $\cov(\omega_0,\omega_n)=0$ as soon as $|n|>q$ for some fixed $q\in\mathbb{N}^*$). The annealed system is then related to a pinning model for a particular Markov renewal process, and the annealed critical curve is given in terms of a Perron-Frobenius eigenvalue. The following high-temperature behaviour is also proved: $$h_c^a(\beta) \stackrel{\beta\searrow0}{\sim} -C(\rho,K(\cdot))\frac{\beta^2}{2}$$ where $$C(\rho,K(\cdot)) = 1 + 2 \sum_{n\geq1} \rho_n P(n\in\tau).$$ The question of disorder irrelevance is ongoing work.

\subsection{Two models with markov disorder}

In this paper we continue to investigate the case of correlated disorder by studying the case when $\omega$ is given by the functional of a homogeneous Markov chain with finite state space. The question of a disordered pinning model with Markov dependence was raised in \cite[p.204]{DenH_Book} (in the context of copolymers near a linear selective interface as well). We will look at two different models. Motivations will be discussed in the next subsection.

\subsubsection{Model A.}

Let $X=(X_n)_{n\geq0}$ be a homogeneous irreducible Markov chain on a finite state space $\Sigma$. Its transition matrix will be denoted by $Q$ and its initial distribution is its invariant distribution $\mu_0$. Note that in this case, $X$ is ergodic. The disorder sequence will be given by $$\omega_n = f(X_n)$$ where $f:\Sigma\mapsto\mathbb{R}$ is such that $\mathbb{E}(\omega_0) (= \mu_0(f(X_0))) = 0$. For this first model we will show that the annealed free energy satisfies an implicit equation and that the annealed critical curve can be expressed as the Perron-Frobenius eigenvalue of a positive matrix which depends both on the kernel $K$ and the matrix $Q$. Moreover, we will see along the proof that the annealed system at the annealed critical curve is equivalent to the pinning of a Markov renewal process. We mention that such tools were previously used in the study of pinning with periodic inhomogeneities (see \cite{CaGiZa07}; indeed, our model includes periodic sequences). As noticed in the previous subsection, the annealed critical curve provides a lower bound on $h_c(\beta)$, but the issue of disorder (ir)relevance is open.

\subsubsection{Model B.}

In model B, the disorder sequence will be a Markov chain taking values in $\{-1,+1\}$, with transition matrix $$Q^{(N)} = \left( \begin{array}{cc}  1-N^{-\gamma}& N^{-\gamma}\\ N^{-\gamma} &1-N^{-\gamma} \end{array}\right),$$ where $\gamma\in(0,1)$, and starting at the invariant distribution $\mu = (1/2, 1/2)$. Its law will be denoted by $\mathbb{P}^{(N)}$, the important difference with model A being that this law now depends on the size $N$ of the system. We will give a motivation for introducing such a model at the end of the following section: at this point, we only notice that the disorder sequence for a system of size $N$ consists roughly of strips of $-1$ and $+1$'s, the size of them being of order $N^{1-\gamma}$. In this model we will determine the limit as $N$ tends to infinity of the (averaged) free energy $\mathbb{E}^{(N)}(F_N(\beta,h))$, in terms of the free energy of the homogeneous model.

\subsection{Motivations. The link with DNA denaturation}

One of the applications  of pinning models is the study of DNA denaturation (or melting), that is the process by which the two strands of a DNA molecule separate as the temperature increases. The two strands are constituted of complementary sequences of nucleotides (A with T, and C with G) and are pinned together by hydrogen bonds, the strength of which depends on the pair of nucleotides (AT bonds are weaker than CG ones). As the temperature increases, entropy wins over energy, and the two strands separate by forming loops (first at AT regions). In our context , the renewal points $(\tau_n)_{n\geq0}$ stands for the sites where the two strands are pinned together and $K(\cdot)$ is the law for the length of loops, whereas $\omega$ stands for the nucleotide sequence (and therefore should be a binary sequence). With this interpretation of the model, the phase transition corresponds to DNA denaturation. Furthermore, it has been shown that DNA sequences display long-range correlations (see \cite{Peng} and \cite{Chen} for instance, and \cite{UnzipDNA},\cite{bubble} for models where this fact is taken into account ). For the link between the pinning model and the Poland-Scheraga model, see \cite[Section 1.4]{GG_Book} and references therein.

\section{Model A. The annealed critical curve}

Before we state our result, we need further notations. Let $M(t,\beta,h)$ be the $\Sigma\times\Sigma$ nonnegative matrix defined by
\begin{equation}\label{M}
M(t,\beta,h)(x,y) = K(t)Q^t(x,y)e^{\beta f(y) + h}
\end{equation}
 and for all $b\geq0$, \begin{equation}\label{A}
 A(b,\beta,h) := \sum_{t\geq1} M(t,\beta,h) \exp(-bt).
\end{equation}
Since $Q$ is irreducible, for all $x$ and $y$ there exists $t$ such that $K(t)Q^t(x,y) > 0$, so $A(b,\beta,h)$ is a positive matrix, and its Perron-Frobenius eigenvalue will be denoted by $\lambda(b,\beta,h)$. Henceforth we will write $\lambda(\beta)$ for $\lambda(0,\beta,0)$.

\subsection{Statement of results}

For model A our result is:
\begin{theo}
 For all $h$ and nonnegative $\beta$,
\begin{equation*}
\frac{1}{N}\log \mathbb{E} Z_{N,\beta,h} \stackrel{N\rightarrow +\infty}{\longrightarrow} F^a(\beta,h) \geq0.
\end{equation*}
where $F^a(\beta,h)$ is solution of the implicit equation
\begin{equation}\label{implicit}
\lambda(F^a(\beta,h),\beta,h) = 1
\end{equation}
if $\lambda(0,\beta,h)>1$ and $F^a(\beta,h)=0$ otherwise. The annealed critical curve is
\begin{equation*}
 h_c^a(\beta) := \sup\{h | F^a(\beta,h) = 0\} = - \log \lambda(\beta)
\end{equation*}
\end{theo}

Notice that (\ref{implicit}) is the analog of the implicit equation for the annealed free energy in the i.i.d. case (which sums up to a homogeneous free energy):
\begin{equation*}
\sum_{n\geq 1} K(n) \exp(-F^a(\beta,h)n) = \exp(-h-\log\Lambda(\beta)).
\end{equation*}

\subsection{Proof}

\begin{proof}
We start by decomposing the quenched partition function according to the number of renewal points before $N$:
\begin{equation*}
 Z_{N,\beta,h,X} =  \sum_{k=1}^N \sum_{\substack{0 =: t_0 < t_1 < \ldots \\ \ldots < t_{k-1} < t_k := N}} e^{kh + \beta(f(X_{t_1})+\ldots+f(X_{t_k}))}\prod_{i=1}^k K(t_i - t_{i-1}).
\end{equation*}
Averaging on $X$ gives for the annealed partition function:
\begin{align*}
 \mathbb{E}Z_{N,\beta,h} = & \sum_{k=1}^N \sum_{\substack{0 =: t_0 < t_1 < \ldots \\ \ldots < t_{k-1} < t_k := N}} \sum_{x_0,x_1,\ldots,x_k \in \Sigma} e^{kh + \beta(f(x_1)+\ldots+f(x_k))}\times \ldots\\& \times \mu_0(x_0) Q^{t_1}(x_0,x_1) Q^{t_2-t_1}(x_1,x_2) \ldots Q^{t_k-t_{k-1}}(x_{k-1},x_k)\\ &\times K(t_1) K(t_2-t_1) \ldots K(t_k-t_{k-1}).
\end{align*}
Recall notation (\ref{M}) and notice that:
\begin{equation*}
 \mathbb{E}Z_{N,\beta,h} =  \sum_{k=1}^N \sum_{\substack{0 =: t_0 < t_1 < \ldots \\ \ldots < t_{k-1} < t_k := N}} \sum_{\substack{x_0,x_1\\ \ldots,x_k \in \Sigma}} \mu_0(x_0) \prod_{i=1}^k M(t_i-t_{i-1},\beta,h)(x_{i-1},x_i).
\end{equation*}
We will denote by $\xi(b,\beta,h)$ an eigenvector (defined up to a scalar, with positive entries) associated to $\lambda(b,\beta,h)$, the Perron-Frobenius eigenvalue of the  positive matrix $A(b,\beta,h)$ defined in (\ref{A}). The dependence on $h$ of these quantities is simple: we have $\lambda(b,\beta,h) = e^h \lambda(b,\beta,0)$ and since $\xi$ is defined up to a scalar, the \textit{h} in $\xi(b,\beta,h)$ is irrelevant, so we will simply write $\xi(b,\beta)$. From the definition of $\lambda$ and $\xi$, we have for all $b\geq0$ and all $x\in \Sigma$,
\begin{equation*}
 \sum_{y\in E} \frac{A(b,\beta,h)(x,y)\xi(b,\beta)(y)}{\lambda(b,\beta,h)\xi(b,\beta)(x)} = 1,
\end{equation*}
which also writes
\begin{equation}\label{step1}
 \sum_{t\geq1} \sum_{y\in E} \frac{M(t,\beta,h)(x,y)e^{-bt}\xi(b,\beta)(y)}{\lambda(b,\beta,h)\xi(b,\beta)(x)} = 1.
\end{equation}

Suppose first that $\lambda(0,\beta,h)>1$. Since $\lambda(F,\beta,h)$ is a continuous (in fact, smooth), strictly decreasing function of $F$, such that $\lambda(F,\beta,h)$ tends to $0$ as $F$ tends to $+\infty$, there exists $F^a = F^a(\beta,h) > 0$ such that $\lambda(F^a,\beta,h)=1$. Let
\begin{equation*}
 p(x,y,t) := M(t)(x,y)e^{-F^at}\frac{\xi(y)}{\xi(x)}
\end{equation*}
(we omit $\beta$ and $h$ for a moment) so that for all $x\in\Sigma$,
\begin{equation}\label{kernel}
 \sum_{t\geq1} \sum_{y\in \Sigma} p(x,y,t) = 1.
\end{equation}
What  equation (\ref{kernel}) means is precisely that $p$ is the kernel of a Markov renewal process (also called semi-Markov kernel): let $\overline{\tau}$ be a process on the integers, defined by
\begin{equation}\label{tau_barre}
 \overline{\tau}_n = \sum_{k=1}^n \overline{T}_k
\end{equation}
where the interarrival times are given by the following Markov chain in $\Sigma\times\mathbb{N}$ (the value of $\overline{T}_0$ is not important):
\begin{equation*}
 \mathbb{P}((\overline{X}_0,\overline{T}_0)=(x_0,t)) = \mu_0(x_0)\delta_{0}(t),
\end{equation*}
\begin{equation*}
 \mathbb{P}((\overline{X}_{n+1},\overline{T}_{n+1})=(y,t)|(\overline{X}_n,\overline{T}_n)=(x,s)) = p(x,y,t).
\end{equation*}
The process $\overline{X}$ is called the modulating chain of the Markov renewal process $\overline{\tau}$. From (\ref{kernel}) and (\ref{step1}), one obtains
\begin{align*}
 \mathbb{E} Z_{N,\beta,h} = e^{F^aN} \sum_{k=1}^N \sum_{\substack{0 =: t_0 < t_1 < \ldots \\ \ldots < t_{k-1} < t_k := N}} \sum_{x_0,x_1,\ldots,x_k \in \Sigma}& \frac{\xi(x_0)}{\xi(x_k)} \mu_0(x_0)\\&\times \prod_{i=1}^k p(x_{i-1},x_i,t_i-t_{i-1}).
\end{align*}
Notice that
\begin{equation*}
 \mathbb{P}(N\in\overline{\tau}) = \sum_{k=1}^N \sum_{\substack{0 =: t_0 < t_1 < \ldots \\ \ldots < t_{k-1} < t_k := N}} \sum_{x_0,x_1,\ldots,x_k \in \Sigma} \mu_0(x_0) \prod_{i=1}^k p(x_{i-1},x_i,t_i-t_{i-1}),
\end{equation*}
so we have 
\begin{equation*}
 0 < c := \min_{x,y\in\Sigma} \frac{\xi(x)}{\xi(y)} \leq \frac{\mathbb{E}Z_{N,\beta,h}}{e^{F^aN} \mathbb{P}(N\in\overline{\tau})} \leq \max_{x,y\in\Sigma} \frac{\xi(x)}{\xi(y)} =: C.
\end{equation*}
Since $\mathbb{P}(N\in\overline{\tau})\leq 1$, it is sufficient to prove that $\liminf_{N\rightarrow +\infty} P(N\in\overline{\tau}) > 0$. Let $x$ be an element of $\Sigma$ and $(\overline{\tau}_n^{x})_{n\geq 0}$ the successive return times of the Markov renewal process $\overline{\tau}$ where the underlying Markov chain is in the state $x$. Then $\overline{\tau}^{x}$ is a (delayed) renewal process with finite mean inter-arrival time $m^x = \sum_{y\in\Sigma} m_y \frac{\mu_0(y)}{\mu_0(x)}$ where $m_y := E(T_1 | X_0 = y)$. Thus we have by the renewal theorem:
\begin{equation*}
 P(N\in\overline{\tau}) \geq P(N\in\overline{\tau}^x) \rightarrow \frac {1}{m^x} >0.
\end{equation*}

\par We are left with the case $\lambda(0,\beta,h) \leq 1$. Note that in any case,
\begin{equation*}
 Z_{N,\beta,h,X} \geq \exp(h+\beta f(X_N)) K(N)
\end{equation*}
so $\liminf_{N\rightarrow\infty} \frac{1}{N} \log \mathbb{E}Z_{N,\beta,h} \geq 0$. Suppose that $\lambda(0,\beta,h) = 1$. Then one can repeat the same arguments as in the case $\lambda(0,\beta,h)>1$, except $F^a = 0$, and one finds $\mathbb{E}Z_{N,\beta,h}\leq C \times \mathbb{P}(N\in\overline{\tau})$, so $\lim_{N\rightarrow} \frac{1}{N}\log\mathbb{E}Z_{N,\beta,h} = 0$. If $\lambda(0,\beta,h)<1$, we also take $F^a=0$ in (\ref{kernel}) and $p$ is now a sub-probability. Again, we have $\mathbb{E}Z_{N,\beta,h}\leq C \times \mathbb{P}(N\in\overline{\tau})$, where $\overline{\tau}$ is a transient Markov renewal process, and $\lim_{N\rightarrow +\infty} \frac{1}{N}\log\mathbb{E}Z_{N,\beta,h} = 0$.

We have now proven the first point of the theorem. The second point is an immediate consequence:
\begin{align*}
 F_a = F_a(\beta,h) > 0 &\Leftrightarrow \lambda(0,\beta,h) > 1 \\
& \Leftrightarrow e^h \lambda(0,\beta,0) > 1 \\
& \Leftrightarrow h >  -\log \lambda(0,\beta,0) = -\log \lambda(\beta)
\end{align*}
which means that
\begin{equation*}
 h_c^a(\beta)  = - \log \lambda(\beta).
\end{equation*}
\end{proof}

\subsection{Examples}

\subsubsection{Moving averages}
Let $\varepsilon = (\varepsilon_n)_{n\in\mathbb{Z}}$ be a sequence of i.i.d. random variables with values in a finite state space $A$. Let $(a_0,\cdots,a_q)$ be in $\mathbb{R}^{q+1}$ and define $\omega$ a $q$-th order moving average by $\omega_n = a_0 \varepsilon_n + a_1 \varepsilon_{n-1} + \cdots + a_q \varepsilon_{n-q}$. These locally dependent variables are in the scope of this paper since $\omega_n = f(X_n)$ where $X_n = (\varepsilon_{n-q},\cdots,\varepsilon_n)$ is indeed a Markov chain on $\Sigma := A^{q+1}$ and $f$ is a function on $\Sigma$: $f(x_0,x_1,\cdots,x_q) = a_0 x_q + a_1 x_{q-1} + \cdots + a_q x_0$.

\par Take for instance $q=1$, $A=\{-1,+1\}$, and $\mathbb{P}(\epsilon_n = 1)=\mathbb{P}(\epsilon_n = -1)=1/2$. Then we have to consider a markov chain in $$A^2=\{(-1,-1),(-1,+1),(+1,-1),(+1,+1)\}$$ with transition matrix $$Q = \left(\begin{array}{cccc}1/2& 1/2 & 0& 0\\0&0&1/2&1/2 \\1/2& 1/2 & 0& 0 \\0&0&1/2&1/2\end{array} \right).$$ Since for all $t\geq2$, $$Q^t = \left(\begin{array}{cccc}1/4&1/4&1/4&1/4\\1/4&1/4&1/4&1/4\\1/4&1/4&1/4&1/4\\1/4&1/4&1/4&1/4\end{array} \right),$$ the matrix $A(0,\beta,0)$ can be easily computed, and its Perron-Frobenius eigenvalue turns out to be $$\lambda(\beta)=\cosh(a_0\beta)\cosh(a_1\beta)\left(1+K(1)(\frac{\cosh((a_0+a_1)\beta)}{\cosh(a_0\beta)\cosh(a_1\beta)}-1)\right).$$ One can check that we find the same result by using the method suggested in \cite{Poisat_frc} for $q=1$.

\subsubsection{A chain in $\{-1,+1\}$}

We consider a particular case where the annealed critical curve can be computed. Suppose that $\omega$ is a Markov chain consisting of $+1$ and $-1$'s, with transition matrix ($0\leq \epsilon < 1$)
\begin{equation*}
 Q = \left(\begin{array}{cc}
      \epsilon & 1-\epsilon\\ 1-\epsilon & \epsilon
     \end{array}\right)
\end{equation*}
and invariant probability measure $\mu = \frac{\delta_{-1}+ \delta_{+1}}{2}$. If $\epsilon = 0$, then the sequence is periodic (see \cite{CaGiZa07}) whereas $\epsilon = 1/2$ is the i.i.d. setting. The eigenvalues of $Q$ are $1$ and $2\epsilon -1$ with respective eigenvectors $(1,1)$ and $(1,-1)$, so for all $t\geq1$,
\begin{equation}
 Q^t = \frac{1}{2} \left( \begin{array}{cc} 1 + (2\epsilon -1)^t & 1 - (2\epsilon -1)^t  \\ 1 - (2\epsilon -1)^t &  1 + (2\epsilon -1)^t \end{array} \right)
\end{equation}
and
\begin{equation*}
 A(0,\beta,0) = \left( \begin{array}{cc} e^{-\beta}p(\epsilon) & e^{\beta}(1-p(\epsilon)) \\ e^{-\beta}(1-p(\epsilon))  &  e^{\beta}p(\epsilon)\end{array} \right)
\end{equation*}
where $p(\epsilon) := \sum_{t\geq 1} K(t) \frac{1+(2\epsilon -1)^t}{2}$. Then, we find by computing the Perron-Frobenius eigenvalue of $A$:
\begin{equation*}
 h_c^a(\beta) = -\log\left(p(\epsilon) \cosh(\beta) + \sqrt{p(\epsilon)^2\cosh^2(\beta) - 2p(\epsilon) +1}  \right)
\end{equation*}
and one can check the consistency with the periodic and i.i.d. cases.

\section{Model B. The phase diagram}

Consider the example of the Markov chain in $\{-1,+1\}$ given above, and notice that $h_c^a(\beta)$ is increasing in $\epsilon$ if $\epsilon$ is in $[1/2,1]$. By making $\epsilon$ tend to $1$ (i.e. increasing correlations of disorder), we make $h_c^a(\beta)$ tend to $-\beta$, which is the minimum of all possible values for the critical curve (simply because $\beta\omega \leq \beta$). Of course we do not want to make $\epsilon$ equal to $1$, which would lead to a homogeneous model with reward equal to $-1$ or $1$ with probability one half. That is why we introduce model B, where $\epsilon$ tends to 1 (while staying positive, so that asymptotically there is the same number of $-1$ and $+1$'s) as the size of the system goes to $+\infty$, as a caricature of a disordered model with strong correlations.

\subsection{Statement of the result}

In the following, we will denote by $F(h):=F(0,h)$ the free energy associated to a polymer pinned at a homogeneous interface with reward $h$. In this section we will prove the following convergence:
\begin{theo}
 For all $\beta\geq0$, $h\in\mathbb{R}$,
\begin{equation*}
 \mathbb{E}^{(N)}F_N(\beta,h) \stackrel{N \rightarrow +\infty}{\longrightarrow} F(\beta,h):=\left\{\begin{array}{ccc}
             0 & \hbox{ if } & h \leq -\beta\\
	\frac{F(h+\beta)}{2}	& \hbox{ if } &  -\beta < h < \beta\\
	\frac{F(h+\beta) + F(h-\beta)}{2}	& \hbox{ if } & h \geq \beta\\
            \end{array}\right..
\end{equation*}
\end{theo}

\begin{center}
\begin{figure}\label{phasediagram}
 \centering
 \includegraphics[scale=0.3]{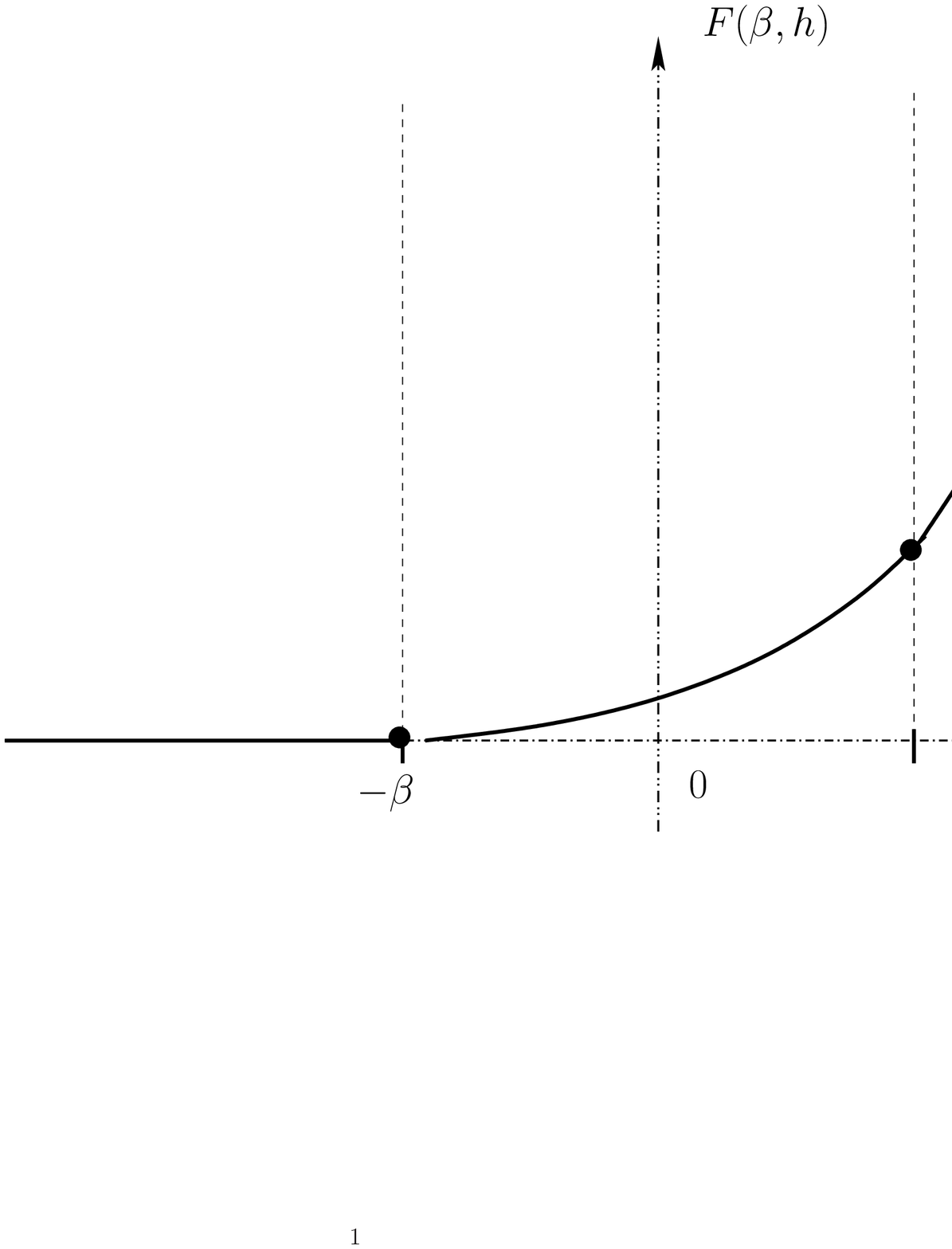}  
 \caption{A picture of the phase diagram at $\beta$ fixed, for model B: $\beta$ and $-\beta$ are the points where $F(\beta,\cdot)$ is not analytic (this is a direct consequence of the non-analiticity of the homogeneous free energy at $h=0$, see \cite{GG_Book}).}
\end{figure}
\end{center}

In this model, we therefore have the full phase diagram picture (see figure \ref{phasediagram}). What the previous statement says is that there is one delocalized phase and two localized phases. For $h\leq -\beta$, the whole interface is repulsive and it is clear that the free energy is zero. If $h>-\beta$ but $h\leq\beta$, then half of the strips are attractive (with reward $h+\beta$), the other half being repulsive ($h-\beta$). Since the strips are large enough, the polymer can be pinned at those attractive strips, so the resulting free energy is half of the free energy associated to the homogeneous reward $h+\beta$. If $h>\beta$, the whole interface is attractive, and the free energy should be the mean of the free energy received on ($+1$)-strips ($F(\beta+h)$) and the one received on ($-1$)-strips ($F(\beta-h)$). Hence it is a phase diagram with two phase transitions, in the sense that there are two points, $-\beta$ and $\beta$, where the free energy function is not analytic.

Before we give the sketch of the proof, we define the following random variables, related to disorder:
\begin{equation*}
\begin{array}{l}
 L_0 := 0\\
L_1 = \inf\{n\geq 1, \omega_n \neq \omega_0\}-1\\
\vdots\\
L_{k+1} = \inf\{n>L_k, \omega_n \neq \omega_{L_{k}+1}\}-1
\end{array}
\end{equation*}
which are the endpoints of the $\pm 1$ strips,
\begin{equation*}
 l_n := L_n - L_{n-1},
\end{equation*}
the lengths of the strips, and
\begin{equation*}
 B_N := \sup\{k\geq0, L_k \leq N\}
\end{equation*}
the number of complete strips between $0$ and $N$.

First we will prove that the theorem is true for a modified version of the model, where the polymer is constrained to visit the endpoints of each strip: its partition function writes
\begin{equation*}
 Z_{N,\beta,h}^c = E\left( \exp(\sum_{n=1}^N (\beta \omega_n + h) \delta_n) (\prod_{i=1}^{B_N} \delta_{L_i}) \delta_N \right),
\end{equation*}
and its free energy,
\begin{equation*}
 F_N^c(\beta,h) = \frac{1}{N}\log Z_{N,\beta,h}^c.
\end{equation*}
Then we will show that one can safely pin the polymer at those endpoints: for each endpoint, the cost is at most polynomial in $N$, and what make things work is that the number of endpoints is of order $N^{1-\gamma} \ll N$.

\subsection{Proof}

In this section, $Z_{N,h}^c$ is the partition function of a polymer pinned at $0$ and $N$, in a homogeneous environment with reward $h$. First we prove:
\begin{lem}
 $\mathbb{E}^{(N)}F^c_N(\beta,h) \stackrel{N \rightarrow +\infty}{\longrightarrow} F(\beta,h).$
\end{lem}
\begin{proof}
 By Markov property we may write:
\begin{equation}\label{dec}
 Z_{N,\beta,h}^c = \left( \prod_{i=1}^{B_N} Z_{l_i,\beta\omega_{L_{i-1}+1}+h}^c \right) Z_{N-L_{B_N},\beta\omega_{L_{B_N}+1}+h}^c
\end{equation}
We will use the following estimates of the partition function for the homogeneous pinning (see \cite{GG_Book}, we do not give the sharpest versions, but these will be enough for our purpose).
If $\epsilon>0$, there exists $c(\epsilon)>0$ such that for all $l\geq1$,
\begin{equation*}
 c(\epsilon) e^{F(\epsilon)l} \leq Z_{l,\epsilon}^c \leq e^{F(\epsilon)l}.
\end{equation*}
If $\epsilon<0$, there exists $c_2(\epsilon)>0$ such that for all $l\geq1$,
\begin{equation*}
 c_2(\epsilon)K(l) \leq Z_{l,\epsilon}^c \leq 1
\end{equation*} and for $\epsilon = 0$,
\begin{equation*}
 1 \geq P(l \in \tau) = Z_{l,0}^c \geq \left\{ \begin{array}{cc} c_3>0 & \hbox{ if } \alpha>1\\ cl^{\alpha-1-\varsigma} & \hbox{ if } \alpha \in (0,1] \end{array}\right.
\end{equation*}
(with $\varsigma>0$ but arbitrarily small).

Let's begin with the case $h\in(-\beta,\beta)$. If we apply the estimates above to the decomposition in (\ref{dec}), we get: (we write $h=-\beta+\epsilon$)
\begin{equation*}
 Z_{N,\beta,h}^c \leq \exp(F(\epsilon)\sum_{n=1}^N \bf{1}_{\{\omega_n=1\}})
\end{equation*}
and
\begin{equation*}
 Z_{N,\beta,h}^c \geq c(\epsilon)^{B_N/2}\exp(F(\epsilon)\sum_{n=1}^N \mathbf{1}_{\{\omega_n=1\}}) \times c_2(\epsilon)^{B_N/2} (\prod_{i=1, \omega_{L_{i-1}}=-1}^{B_N}l_i^{-\theta})
\end{equation*}
with some power $\theta>0$. In the last product, each factor is larger than $N^{-\theta}$. Taking the log and dividing by $N$, we obtain:
\begin{equation*}
\frac{B_N(\log(c(\epsilon)c_2(\epsilon))-\theta \log N)}{2N} \leq \frac{1}{N}\log Z_{N,\beta,h}^c - \frac{F(\epsilon)}{N}\sum_{n=1}^N \mathbf{1}_{\{\omega_n=1\}}) \leq 0
\end{equation*}
The result follows by taking the expectation with respect to $\mathbb{P}^{(N)}$ and letting $N$ go to infinity, since
\begin{equation*}
 \mathbb{E}^{(N)} B_N = \sum_{n=1}^N \mathbb{P}^{(N)}(\omega_n \neq \omega_{n+1}) = N \times N^{-\gamma} \ll N.
\end{equation*}
The proof is the same for $h\geq\beta$.
\end{proof}

We now want to prove that $\mathbb{E}^{(N)}F_N(\beta,h)$ and $\mathbb{E}^{(N)}F^c_N(\beta,h)$ have the same limit. Since, obviously,
\begin{equation*}
 Z_{N,\beta,h} \geq Z_{N,\beta,h}^c
\end{equation*}
we have
\begin{equation*}
 \liminf \mathbb{E}^{(N)}F_N(\beta,h) \geq \lim \mathbb{E}^{(N)}F_N^c(\beta,h) = F(\beta,h)
\end{equation*}
so the only thing we need to prove is
\begin{lem}
 $\limsup \mathbb{E}^{(N)}F_N(\beta,h) \leq \lim \mathbb{E}^{(N)}F^c_N(\beta,h) (= F(\beta,h))$
\end{lem}
\begin{proof}
In this proof we use the following notation:
\begin{equation*}
 \overline{K}(n) = \sum_{l>n} K(l)
\end{equation*}
with the fact that $\overline{K}(n) \sim n K(n)$. First we prove that the cost of pinning the polymer at one arbitrary point $n\in\{1,\ldots,N\}$, for $\beta$ and $h$ fixed, is at most polynomial in $N$. If $n$ is not a contact point, then we can decompose the partition on the last point visited before $n$ (say $n-a$) and the first point visited after $n$ (say $n+b$). One obtains
\begin{align*}
&E(\exp(H_N)\delta_N(1-\delta_n))\\ &= \sum_{a=1}^n \sum_{b=1}^{N-n} Z_{n-a,\beta,h,\omega}e^{\beta\omega_{n+b}+h}K(a+b)Z_{N-n-b,\beta,h,\theta^{n+h}\omega}\\
& = \sum_{a=1}^n \sum_{b=1}^{N-n} Z_{n-a,\beta,h}K(a)e^{\beta\omega_n+h,\omega}K(b)e^{\beta\omega_{n+b}+h}Z_{N-n-b,\beta,h,\theta^{n+h}\omega} C(a,b,n)
\end{align*}
where $(\theta\omega)_n = \omega_{n+1}$ and
\begin{align*}
C(a,b,n) :=e^{-(\beta\omega_n+h)}\frac{K(a+b)}{K(a)K(b)}\leq e^{\beta+|h|}\frac{\overline{K}(b)}{K(a)K(b)}&\leq Cba^{1+\alpha+\varsigma} \\&\leq C N^{2+\alpha+\varsigma}.
\end{align*}
with some positive $\varsigma$. Therefore, we have $$Z_{N,\beta,h} \leq (1+CN^{2+\alpha+\varsigma}) E(\exp(H_N)\delta_N\delta_n).$$ Repeating the operation $B_N$ times, we get
\begin{equation*}
 Z_{N,\beta,h} \leq (1+CN^{2+\alpha+\varsigma})^{B_N}\times Z_{N,\beta,h}^c.
\end{equation*}
taking the log, dividing by $N$, and averaging over $\omega$ (again,  we use $\mathbb{E}^{(N)}B_N = N^{1-\gamma}$), we get the result.
\end{proof}

\subsection{A remark on the general case}

The result can actually be generalized to an arbitrary finite number of states. We give it here for completeness. However, the generalization of the proof given for two-state Markov chains is quite straightforward, so we have chosen not to reproduce it here.

\begin{theo}
Suppose that $\mathbb{P}^{(N)}$ is the law of a finite state stationary (real) Markov chain with transition matrix $$Q^{(N)} = \Id + N^{-\gamma}(Q-\Id),$$ where $Q$ is another irreducible transition matrix on the same state space. If $\mu$ is the invariant distribution of $Q$, then it also the invariant distribution of $Q^{(N)}$, and $$\mathbb{E}^{(N)}F_N(\beta,h) \stackrel{N\rightarrow +\infty}{\longrightarrow} F(\beta,h) = \mathbb{E}_{\mu}\left( F(h+\beta\omega_0)\right) $$
\end{theo}

Suppose that the states of $\omega$ are $x_1<x_2<\ldots<x_n$. What the last result tells us is that for $\beta>0$ fixed, there is a phase transition (in the sense that the free energy is not analytic) at the points $-\beta x_n < -\beta x_{n-1} < \ldots < -\beta x_1$. If $h\leq -\beta x_n$ then $F(\beta,h)=0$ and if $-\beta x_i < h \leq -\beta x_{i-1}$ then $F(\beta,h) = \mu_n F(h+\beta x_n) + \ldots + \mu_{i+1}F(h+\beta x_{i+1}) + \mu_i F(h+\beta x_i)$. Therefore, one obtains a phase diagram with a multistep depinning transition.

\section*{Acknowledgements}
The author is grateful to Hubert Lacoin for a discussion on the second model.

\bibliographystyle{spmpsci}      
\bibliography{references}   

\def\polhk\#1{\setbox0=\hbox{\#1}{{\o}oalign{\hidewidth
  {\l}ower1.5ex\hbox{`}\hidewidth\crcr\unhbox0}}}
\begin{thebibliography}{10}
\providecommand{\url}[1]{{#1}}
\providecommand{\urlprefix}{URL }
\expandafter\ifx\csname urlstyle\endcsname\relax
  \providecommand{\doi}[1]{DOI~\discretionary{}{}{}#1}\else
  \providecommand{\doi}{DOI~\discretionary{}{}{}\begingroup
  \urlstyle{rm}\Url}\fi

\bibitem{Alexander_Sido}
Alexander, K.S.: The effect of disorder on polymer depinning transitions.
\newblock Comm. Math. Phys. \textbf{279}(1), 117--146 (2008).
\newblock \doi{10.1007/s00220-008-0425-5}.
\newblock \urlprefix\url{http://dx.doi.org/10.1007/s00220-008-0425-5}

\bibitem{Alexander_quenched}
Alexander, K.S., Zygouras, N.: {Quenched and annealed critical points in
  polymer pinning models.}
\newblock Commun. Math. Phys. \textbf{291}(3), 659--689 (2009).
\newblock \doi{10.1007/s00220-009-0882-5}

\bibitem{Alexander_loop_exp_one}
Alexander, K.S., Zygouras, N.: {Equality of critical points for polymer
  depinning transitions with loop exponent one.}
\newblock Ann. Appl. Probab. \textbf{20}(1), 356--366 (2010).
\newblock \doi{10.1214/09-AAP621}

\bibitem{UnzipDNA}
Allahverdyan, A.E., Gevorkian, Z.S., Hu, C.K., Wu, M.C.: Unzipping of dna with
  correlated base sequence.
\newblock Phys. Rev. E \textbf{69}(6), 061,908 (2004).
\newblock \doi{10.1103/PhysRevE.69.061908}

\bibitem{CaGiZa07}
Caravenna, F., Giacomin, G., Zambotti, L.: {A renewal theory approach to
  periodic copolymers with adsorption.}
\newblock Ann. Appl. Probab. \textbf{17}(4), 1362--1398 (2007).
\newblock \doi{10.1214/105051607000000159}

\bibitem{Cheliotis_DenH}
Cheliotis, D., {den Hollander}, F.: Variational characterization of the
  critical curve for pinning of random polymers (2010).
\newblock \urlprefix\url{http://arxiv.org/abs/1005.3661}

\bibitem{Chen}
Chen, X.Y., Bao, L.J., Mo, J.Y., Wang, Y.: Characterizing long-range
  correlation properties in nucleotide sequences.
\newblock Chinese Chemical Letters Vol. 14 \textbf{14}(5), 503--504 (2003)

\bibitem{DenH_Book}
{den Hollander}, F.: {Random polymers. {\'E}cole d'{\'E}t{\'e} de
  Probabilit{\'e}s de Saint-Flour XXXVII -- 2007.}
\newblock {Lecture Notes in Mathematics 1974. Berlin: Springer. xiii, 258~p.
  EUR~43.95 } (2009).
\newblock \doi{10.1007/978-3-642-00333-2}

\bibitem{Derrida_al_relevance}
Derrida, B., Giacomin, G., Lacoin, H., Toninelli, F.L.: {Fractional moment
  bounds and disorder relevance for pinning models.}
\newblock Commun. Math. Phys. \textbf{287}(3), 867--887 (2009).
\newblock \doi{10.1007/s00220-009-0737-0}

\bibitem{GG_Book}
Giacomin, G.: Random polymer models.
\newblock Imperial College Press, London (2007)

\bibitem{GG_Overview}
Giacomin, G.: Renewal sequences, disordered potentials, and pinning phenomena
  (2008).
\newblock \urlprefix\url{arXiv.org:0807.4285}

\bibitem{GG_T_L_Marginal}
Giacomin, G., Toninelli, F., Lacoin, H.: {Marginal relevance of disorder for
  pinning models.}
\newblock Commun. Pure Appl. Math. \textbf{63}(2), 233--265 (2010).
\newblock \doi{10.1002/cpa.20301}

\bibitem{bubble}
Jeon, J.H., Park, P.J., Sung, W.: The effect of sequence correlation on bubble
  statistics in double-stranded dna.
\newblock The Journal of Chemical Physics \textbf{125} (2006)

\bibitem{Lacoin_martingale}
Lacoin, H.: The martingale approach to disorder irrelevance for pinning models
  (2010).
\newblock \urlprefix\url{http://arxiv.org/abs/1002.4753}

\bibitem{Peng}
Peng, C.K., Buldyrev, S.V., Goldberger, A.L., Havlin, S., Sciortino, F.,
  Simons, M., Stanley, H.E.: Long-range correlations in nucleotide sequences.
\newblock Nature \textbf{356}, 168--170 (1992)

\bibitem{Poisat_frc}
Poisat, J.: Pinning of a polymer on a disordered line with finite range
  correlations: the annealed critical curve (2010).
\newblock \urlprefix\url{http://arxiv.org/abs/0903.3704}

\bibitem{Fabio_replica}
Toninelli, F.L.: {A replica-coupling approach to disordered pinning models.}
\newblock Commun. Math. Phys. \textbf{280}(2), 389--401 (2008).
\newblock \doi{10.1007/s00220-008-0469-6}

\bibitem{Toninelli_Survey}
Toninelli, F.L.: Localization transition in disordered pinning models.
\newblock In: Methods of Contemporary Mathematical Statistical Physics, Lecture
  Notes in Mathematics, pp. 129--176 (2009)

\end{thebibliography}
\end{document}